\documentclass{amsart}

\def\bP{{\mathbf P}}
\def\bQ{{\mathbf Q}}
\def\bN{{\mathbf N}}
\def\bX{{\mathbf X}}
\def\bV{{\mathbf V}}
\def\bW{{\mathbf W}}
\def\bC{{\mathbf C}}
\def\bE{{\mathbf e}}

\theoremstyle{theorem}
\newtheorem*{theorem}{Theorem}

\theoremstyle{corollary}
\newtheorem*{corollary}{Corollary}

\theoremstyle{claim}
\newtheorem*{claim}{Claim}

\theoremstyle{definition}

\begin{document}

\title[Elementary Proof Of The Gauss-Bonnet Theorem In $\Bbb R^3$]{Elementary Proofs Of The Gauss-Bonnet Theorem And Other Integral Formulas In $\Bbb R^3$}
\date{September 15, 2015}
\author{Daniel Mayost}
\address{Office Of The Superintendent Of Financial Institutions, Canada}
\email{Daniel.Mayost@osfi-bsif.gc.ca}
\maketitle

\begin{abstract}
For a compact differentiable surface with boundary embedded in $\Bbb R^3$, we give simple proofs of the Gauss-Bonnet theorem, Poincar\'{e}-Hopf theorem, and several other integral formulas.  We complete all of the proofs without using fundamental or differential forms. 
\end{abstract}

\section{Introduction}

In this note we prove several integral formulas involving the principal curvatures of a surface, including the Gauss-Bonnet theorem.  All of the proofs are elementary in the sense that we do not place any co-ordinates on the surface, and hence we do not need to invoke the theory of fundamental or differential forms.  The results will follow from applying Stokes' theorem directly to the surface, and studying objects associated with the surface that are co-ordinate invariant.  In particular, the principal curvatures of the surface will arise as the directional derivatives of the normal vector to the surface.

\section{Preliminaries}

All of the identities we will prove are for a compact, oriented $C^\infty$ surface $M$ with boundary $\partial M$ embedded in $\Bbb R^3$ with volume element $dA$.  At each point on $M$ we denote the position vector by $\bX$ and the normal vector by $\bN$.  Our results will follow from this variation of Stokes' Theorem:    

\begin{claim}
Let $f$ and $g$ be differentiable functions that map an $\Bbb R^3$-neighborhood of $M$ to $\Bbb R$.  If
$\bP$ and $\bQ$ are orthonormal vector fields on $M$ that are not necessarily continuous, and are oriented so that $\bP\times\bQ=\bN$ everywhere on $M$, then:
\begin{equation*}
\int_{\partial M}f\, dg=\int_M\left[
 \left(\nabla_{\bP}f\right)\left(\nabla_{\bQ}g\right)
-\left(\nabla_{\bQ}f\right)\left(\nabla_{\bP}g\right)
\right]\,dA\end{equation*}
\end{claim}

\begin{proof}
We start with:
\begin{equation*}
\int_{\partial M}f\, dg=\int_{\partial M} f\,\nabla g\cdot d\bX.\end{equation*}
Using Stokes' Theorem and the identity:
\begin{equation*} \text{curl}\,(f\nabla g)=\nabla f\times\nabla g\end{equation*}
gives:
\begin{align*}
\int_{\partial M}f\,dg&=
\int_M \nabla f\times \nabla g\cdot\bN\,dA\\
&=\int_M
\nabla f\cdot\nabla g\times\bN\,dA.\end{align*}
Since $\bP$ and $\bQ$ form an orthonormal basis for the tangent plane at any point on $M$, any tangent vector can be written as the sum of its projection onto $\bP$ and its projection onto $\bQ$.  Therefore:
\begin{align*}
\nabla g\times\bN &= (\nabla g\times\bN\cdot\bP)\bP
+(\nabla g\times\bN\cdot\bQ)\bQ\\
&= (\nabla g\cdot\bN\times\bP)\bP
+(\nabla g\cdot\bN\times\bQ)\bQ\\
&= (\nabla g\cdot\bQ)\bP
-(\nabla g\cdot\bP)\bQ\\
&= \left(\nabla_{\mathbf Q}g\right)\bP-\left(\nabla_{\mathbf P}g\right)\bQ.\end{align*}
We finally get:
\begin{align*}
\int_{\partial M}f\, dg&=
\int_M\left[
  (\nabla_{\bQ}g)(\nabla f\cdot\bP)
- (\nabla_{\bP}g)(\nabla f\cdot\bQ)
\right]\,dA\\
&=\int_M\left[
 (\nabla_{\bP}f)(\nabla_{\bQ}g)
-(\nabla_{\bQ}f)(\nabla_{\bP}g)
\right]\,dA.
\end{align*}
\end{proof}

In what follows we will assume that $\bP$ and $\bQ$ are orthonormal principal directions on $M$ with corresponding principal curvatures $\kappa_1$ and $\kappa_2$, and that they are oriented so that $\bP\times\bQ=\bN$ everywhere on $M$.  We will use the subscripts $p$ and $q$ to denote the directional derivatives of a function in the directions $\bP$ and $\bQ$, respectively.  An immediate corollary of the above claim is that if $\bV$ and $\bW$ are vector fields on $\Bbb R^3$ then:
\begin{equation}
\int_{\partial M}\bV\cdot d\bW
=\int_M \left( \bV_p\cdot\bW_q
- \bV_q\cdot\bW_p \right)\,dA.\label{Ur}
\end{equation}
Since $\bP$ and $\bQ$ are not required to be continuous in the above equation, we will not run into difficulties if $M$ contains umbilical points.  With these definitions we have:
\begin{equation}
\bN_p=-\kappa_1\bP,\qquad
\bN_q=-\kappa_2\bQ.\label{Rod}\end{equation}
Finally, we define the mean curvature $H$ and Gaussian curvature $K$ at points on $M$ by:
\begin{equation*}
H=\frac12(\kappa_1+\kappa_2),\quad
K=\kappa_1\kappa_2.\end{equation*}

\section{Two Curvature Identities}

We begin with two simple consequences of identity~\eqref{Ur}: 

\begin{claim}
If $\bV$ is a vector field on $\Bbb R^3$ then:
\begin{align}
\int_{\partial M}(\bV\times\bN)\cdot d\bX&=
-\int_M\left( \bV_p\cdot\bP+\bV_q\cdot\bQ+2H\,\bV\cdot\bN \right)\,dA\label{hid}\\
\int_{\partial M}(\bV\times\bN)\cdot d\bN&=
\int_M\left(\kappa_2\,\bV_p\cdot\bP+\kappa_1\,\bV_q\cdot\bQ+2K\,\bV\cdot\bN\right)\,dA.
\label{kid}\end{align}
\end{claim}
\begin{proof}
Using equations~\eqref{Ur} and~\eqref{Rod} we get:
\begin{align*}
&\int_{\partial M}(\bV\times\bN)\cdot d\bX\\
=&
\int_M\left(
 \bV_p\times\bN\cdot\bQ-\kappa_1\bV\times\bP\cdot\bQ
-\bV_q\times\bN\cdot\bP+\kappa_2\bV\times\bQ\cdot\bP\right)\,dA\\
=&\int_M\left(
 \bV_p\cdot\bN\times\bQ-\kappa_1\bV\cdot\bP\times\bQ
-\bV_q\cdot\bN\times\bP+\kappa_2\bV\cdot\bQ\times\bP\right)\,dA\\
=&\int_M\left(
-\bV_p\cdot\bP-\bV_q\cdot\bQ-2H\,\bV\cdot\bN\right)\,dA\\
\vspace{1\jot}
&\int_{\partial M}(\bV\times\bN)\cdot d\bN\\
=&\int_M\left(
-\kappa_2\bV_p\times\bN\cdot\bQ+\kappa_1\kappa_2\bV\times\bP\cdot\bQ
+\kappa_1\bV_q\times\bN\cdot\bP-\kappa_1\kappa_2\bV\times\bQ\cdot\bP\right)\,dA\\
=&\int_M\left(
-\kappa_2\bV_p\cdot\bN\times\bQ+\kappa_1\kappa_2\bV\cdot\bP\times\bQ
+\kappa_1\bV_q\cdot\bN\times\bP-\kappa_1\kappa_2\bV\cdot\bQ\times\bP\right)\,dA\\
=&\int_M\left(\kappa_2\,\bV_p\cdot\bP+\kappa_1\,\bV_q\cdot\bQ+2K\,\bV\cdot\bN\right)\,dA.
\end{align*}
\end{proof}

Equation~\eqref{hid} is a generalization of the divergence theorem for surfaces to vector fields on $\Bbb R^3$ that appears for example in \S7 of~\cite{Simon}.  With these two identities we can quickly prove several well-known results. 

Let $\{\bE_1,\bE_2,\bE_3\}$ denote the standard basis for $\Bbb R^3$.  Successively setting $\bV$ equal to these basis elements in~\eqref{hid} and~\eqref{kid} gives the vector identities:
\begin{align*}
\int_{\partial M}\bN\times d\bX&=-2\int_M H\,\bN\,dA\\
\int_{\partial M}\bN\times d\bN&= 2\int_M K\,\bN\,dA.\end{align*}
Setting $\bV$ equal to the vectors $\bE_1\times\bX$, $\bE_2\times\bX$ and $\bE_3\times\bX$ in~\eqref{hid} and~\eqref{kid} gives the vector identities:
\begin{align*}
\int_{\partial M}\bX\times(\bN\times d\bX)&=-2\int_M H\,\bX\times\bN\,dA\\
\int_{\partial M}\bX\times(\bN\times d\bN)&=2\int_M K\,\bX\times\bN\,dA.
\end{align*}
Setting $\bV=\bX$ in~\eqref{hid} and~\eqref{kid} gives Minkowski's formulas (cf.~\cite{SpivFive}, pp. 181-185):
\begin{align*}
\int_{\partial M}\bX\times\bN\cdot d\bX=-2\int_M(1+H\,\bX\cdot\bN)\,dA\\
\int_{\partial M}\bX\times\bN\cdot d\bN=2\int_M(H+K\,\bX\cdot\bN)\,dA.\end{align*}

\section{The Gauss-Bonnet And Poincar{\'e}-Hopf Theorems}

We next prove the Gauss-Bonnet theorem and the Poincar{\'e}-Hopf theorem using~\eqref{kid} and~\eqref{Ur}.  We start with the following simplified version of Liouville's formula (cf. \cite{Schaum}, problem 11.19):

\begin{claim} Suppose that there exists a constant vector $\bC$ with $\|\bC\|=1$ such that $\bC\cdot\bN\ne\pm1$ on $\partial M$.  Let $s$ denote arc length, with the subscript $s$ denoting the derivative with respect to arc length.  Let $\theta$ denote the angle between the unit tangent vector $\bX_s$ to $\partial M$ and the vector $\bC\times\bN$ tangent to $M$.  There on $\partial M$ there holds:
\begin{equation}
\theta_s=\kappa_g-\frac{\bC\cdot\bN}{1-(\bC\cdot\bN)^2}\,\bC\times\bN\cdot\bN_s\label{Liouville}
\end{equation}
where $\kappa_g$ is the geodesic curvature of $\partial M$.  
\end{claim}
\begin{proof}
We use ${\mathbf [}\;{\mathbf ]}$ to denote the triple product of vectors:
\begin{equation*}
[\mathbf v_1\mathbf v_2\mathbf v_3]=\mathbf v_1\cdot\mathbf v_2\times\mathbf v_3=\mathbf v_1\times\mathbf v_2\cdot\mathbf v_3=\det(\mathbf v_1,\mathbf v_2,\mathbf v_3).
\end{equation*} 
We have the relations:
\begin{equation*}
\bX_s\cdot\bX_s=1,\quad\bX_s\cdot\bN=0,\quad\bN\cdot\bN=0.\end{equation*}
Differentiating these with respect to $s$ gives:
\begin{equation*}
\bX_s\cdot\bX_{ss}=0,\quad\bX_{ss}\cdot\bN=-\bX_s\cdot\bN_s,\quad\bN_s\cdot\bN=0.\end{equation*} 
The geodesic curvature $\kappa_g$ is defined by:
\begin{equation*}\kappa_g=[\bX_s \bX_{ss} \bN]\end{equation*}
and the angle $\theta$ satisfies the relation:
\begin{equation*}
\tan\theta=\frac{\bX_s\cdot\bC}{[\bX_s \bC \bN]}.
\end{equation*}  
Differentiating $\tan^{-1}\theta$ gives:
\begin{equation*}
\theta_s=\frac{[\bX_s \bC \bN](\bX_{ss}\cdot\bC)-(\bX_s\cdot\bC)([\bX_{ss} \bC \bN]+[\bX_s \bC \bN_s])}
{[\bX_s \bC \bN]^2+(\bX_s\cdot\bC)^2}.\end{equation*}
By projecting onto the orthonormal basis $\{\bX_s,\bN,\bX_s\times\bN\}$ for $\Bbb R^3$ we get:
\begin{equation*}
\bX_{ss}=-(\bX_s\cdot\bN_s)\bN-\kappa_g\bX_s\times\bN\end{equation*}
\begin{equation*}
\bN_s=(\bX_s\cdot\bN_s)\bX_s+[\bN_s \bX_s \bN]\,\bX_s\times\bN,\quad\bN_s\times\bX_s=[\bN_s \bX_s \bN]\,\bN
\end{equation*}
\begin{equation*}
\|\bC\|^2=(\bC\cdot\bX_s)^2+(\bC\cdot\bN)^2+[\bC\cdot\bX_s\times\bN]^2=1.\end{equation*}
Therefore:
\begin{align*}
\theta_s&=
\frac{[\bX_s \bC \bN](\kappa_g[\bX_s \bC \bN]-(\bC\cdot\bN)(\bX_s\cdot\bN_s))}
{[\bX_s \bC \bN]^2+(\bX_s\cdot\bC)^2}\\
&\qquad+
\frac{\kappa_g(\bX_s\cdot\bC)^2-(\bX_s\cdot\bC)(\bC\cdot\bN_s\times\bX_s)}
{[\bX_s \bC \bN]^2+(\bX_s\cdot\bC)^2}\\
&=\kappa_g-\frac{\bC\cdot\bN}{1-(\bC\cdot\bN)^2}\left([\bX_s \bC \bN](\bX_s\cdot\bN_s)+(\bX_s\cdot\bC)[\bN_s \bX_s \bN]\right)\\
&=\kappa_g-\frac{\bC\cdot\bN}{1-(\bC\cdot\bN)^2}\left((\bX_s\times\bN)\cdot((\bX_s\cdot\bC)\bN_s-(\bX_s\cdot\bN_s)\bC\right)\\
&=\kappa_g-\frac{\bC\cdot\bN}{1-(\bC\cdot\bN)^2}\left((\bX_s\times\bN)\cdot(\bX_s\times(\bN_s\times\bC))\right)\\
&=\kappa_g-\frac{\bC\cdot\bN}{1-(\bC\cdot\bN)^2}[\bC \bN \bN_s]
\end{align*}
\end{proof}

We can now prove:
\begin{theorem}[Gauss-Bonnet]
\begin{equation}
\int_{\partial M}\kappa_g\,ds+\int_MK\,dA=2\pi\chi(M)\end{equation}
\end{theorem}
\begin{proof}
We triangulate $M$ as is done in classical proofs of the theorem (see for example~\cite{Schaum} pp. 242-246).  For our proof, we choose the triangles small enough so that the range of the Gauss map on each triangle lies properly within an open hemisphere.  This will ensure that, for each triangle, we can find a constant vector $\bC$ with $\|\bC\|=1$ and $\bC\cdot\bN\ne\pm1$ everywhere in the triangle.

We integrate~\eqref{Liouville} over the boundary of each triangle.  As in the classical proof, the $\theta_s$ term in~\eqref{Liouville} gives rise to the Euler characteristic term $\chi(M)$ when the exterior angles at the vertices of each triangle are enumerated.  To derive the surface integral we substitute:
\begin{equation*}
\bV=\frac{\bC\cdot\bN}{1-(\bC\cdot\bN)^2}\,\bC\end{equation*}
in~\eqref{kid}.  The resulting integrand over each triangle is then:
\begin{equation*}
\frac{1+(\bC\cdot\bN)^2}{(1-(\bC\cdot\bN)^2)^2}[(\bC\cdot\bP)^2+(\bC\cdot\bQ)^2]K-\frac{2(\bC\cdot\bN)^2}{1-(\bC\cdot\bN)^2}K,\end{equation*}   
which equates to $K$ because $\{\bP,\bQ,\bN\}$ is an orthonormal basis for $\Bbb R^3$ and:
\begin{equation*}
\|\bC\|^2=(\bC\cdot\bP)^2+(\bC\cdot\bQ)^2+(\bC\cdot\bN)^2=1.\end{equation*}
\end{proof}
The Poincare{\'e}-Hopf theorem will follow from the Gauss-Bonnet theorem and the following integral formula:
\begin{claim} If $\bV$ is a vector field on $M$ with $\|\bV\|=1$ then:
\begin{equation*}
\int_{\partial M}\bV\times\bN\cdot d\bV=\int_M K\,dA\end{equation*}
\end{claim} 
\begin{proof}
Using~\eqref{Ur} gives:
\begin{multline*}
\int_{\partial M}\bV\times\bN\cdot d\bV\\
=\int_M(-\kappa_1\bV\times\bP\cdot\bV_q+\kappa_2\bV\times\bQ\cdot\bV_p-2\,\bV_p\times\bV_q\cdot\bN)\,dA.
\end{multline*}
Differentiating the relation $\bV\cdot\bV=1$ gives:
\begin{equation*}
\bV_p\cdot\bV=\bV_q\cdot\bV=\bN\cdot\bV=0,\end{equation*}
which implies that all of $\bV_p$, $\bV_q$ and $\bN$ lie in a plane perpendicular to $\bV$ and the last term in the integral over $M$ is zero.  From the relation $\bV\cdot\bN=0$ we get:
\begin{equation*}
\bV\times\bP=-(\bV\cdot\bQ)\bN,\quad\bV\times\bQ=(\bV\cdot\bP)\bN\end{equation*}
and differentiating $\bV\cdot\bN=0$ gives:
\begin{equation*}
\bV_p\cdot\bN=\kappa_1\bV\cdot\bP,\quad\bV_q\cdot\bN=\kappa_2\bV\cdot\bQ.\end{equation*}
The surface integral therefore reduces to:
\begin{equation*}
\kappa_1\kappa_2[(\bV\cdot\bP)^2+(\bV\cdot\bQ)^2]=K.\end{equation*}
\end{proof}

\begin{corollary}If $\bV$ is a vector field on $M$ that does not vanish on $\partial M$ and has only isolated singularities on $M$ at points in the set $S$ then:
\begin{equation*}
\int_{\partial M}\frac{\bV\times\bN}{\|\bV\|^2}\cdot d\bV
=\int_M K\,dA-2\pi\sum_{s\in S}\mathrm{Index}_s(\bV)
\end{equation*}
\end{corollary} 
\begin{proof}
This follows from:
\begin{equation*}
\frac{\bV\times\bN}{\|\bV\|^2}\cdot d\bV=\frac{\bV}{\|\bV\|}\times\bN\cdot d\left(\frac{\bV}{\|\bV\|}\right)
\end{equation*}
and the fact that near singularities $s$ of $\bV$, the boundary integral approaches $-$``$d\theta$'' for the vector field projected onto the tangent plane to $M$ at $s$.
\end{proof}
Combining this with the Gauss-Bonnet theorem to eliminate the total curvature term gives:
\begin{theorem}[Poincar{\'e}-Hopf]If $\bV$ is a vector field on $M$ satisfying the same conditions as in the above corollary then:
\begin{equation*}
\chi(M)-\sum_{s\in S}\mathrm{Index}_s(\bV)
=\frac1{2\pi}\int_{\partial M}\left(\frac{\bV\times\bN}{\|\bV\|^2}\cdot d\bV+\kappa_g\,ds\right).
\end{equation*}  
\end{theorem}

The above result extends the Poincar{\'e}-Hopf theorem to vector fields that are not perpendicular to $\partial M$.  When $\bV$ is perpendicular to $\partial M$, the boundary integral reduces to zero and gives the traditional result. 

\section{Identities Containing The Difference Of Principal Curvatures}

We conclude by proving two identities that contain a term equal to the difference of the principal curvatures, instead of their sum or product:

\begin{claim}
For any twice differentiable function $F:\Bbb R^3\to\Bbb R$ there holds:
\begin{align*}
\int_{\partial M}\nabla F(\bX)\cdot d\bN&=-\int_M(\kappa_2-\kappa_1)[(\nabla^2 F(\bX))]\bP\cdot\bQ\\
\int_{\partial M}\nabla F(\bN)\cdot d\bX&= \int_M(\kappa_2-\kappa_1)[(\nabla^2 F(\bN))]\bP\cdot\bQ.
\end{align*}
\end{claim}
\begin{proof}
These follow from~\eqref{Ur} and:
\begin{align*}
&\nabla F(\bX)_p=(\nabla^2 F(\bX))\bP &\nabla F(\bN)_p=-\kappa_1(\nabla^2 F(\bN))\bP\\
&\nabla F(\bX)_q=(\nabla^2 F(\bX))\bQ &\nabla F(\bN)_q=-\kappa_2(\nabla^2 F(\bN))\bQ.
\end{align*}
\end{proof}

\end{document}